\documentclass[12pt,a4paper]{amsart}
\usepackage{amssymb,amsxtra}
\usepackage[english,french]{babel}

\newtheorem{thm}{Theorem}[section]

\theoremstyle{definition}

\theoremstyle{remark}

\theoremstyle{plain}

\theoremstyle{remark}

\newtheorem*{example}{Example}
\numberwithin{equation}{section}

\begin{document}

\title{ On a conjecture concerning the fraction  of  invertible  m-times  Persymmetric  Matrices over $\mathbb{F}_{2} $}
\author{Jorgen~Cherly}
\address{D\'epartement de Math\'ematiques, Universit\'e de
    Brest, 29238 Brest cedex~3, France}
\email{Jorgen.Cherly@univ-brest.fr}
\email{andersen69@wanadoo.fr}

\maketitle 
\begin{abstract}
Dans cette article nous annon\c{c}ons que la fraction d'invertible m-fois matrices persym\' etriques sur $ \mathbb{F}_{2}$
est \' egale \' a $\prod_{j=1}^{m}(1-2^{-j}) $

 \end{abstract}

\selectlanguage{english}

\begin{abstract}
In this paper we state that the fraction of invertible m-times persymmetric matrices over  $ \mathbb{F}_{2}$
is equal to  $\prod_{j=1}^{m}(1-2^{-j}) $

 \end{abstract}
 
 \maketitle 
\newpage
\tableofcontents
\newpage

   \section{Some notations concerning  m-times persymmetric matrices over  $ \mathbb{F}_{2}$}
  \label{sec 1}

  Let  $m, s_{1},s_{2},\ldots s_{m},\delta, k$   be positive integers such that 
  $1\leqslant m\leqslant \sum_{j=1}^{m} s_{j}= \delta \leqslant k.$\\
   We can assume without restriction that 
  $1\leqslant s_{1}\leqslant s_{2}\leqslant \ldots \leqslant s_{m} $ since  the rank of a matrix does not change under elementary row operations.
\\
  We denote by  $ \Gamma_{\delta}^{\left[s_{1}\atop{\vdots \atop s_{m}}\right]\times k} $the number  of rank $ \delta$ m- times persymmetric matrices over the finite field  $\mathbb{F}_{2} $  of the below form
\begin{equation}
 \label{eq 1.1}
 \left (  \begin{array} {ccccc}
\alpha  _{1}^{(1)} & \alpha  _{2}^{(1)} & \ldots  &   \alpha_{k-1}^{(1)} &   \alpha_{k}^{(1)}  \\
\alpha  _{2}^{(1)} & \alpha  _{3}^{(1)}  & \ldots   &   \alpha_{k}^{(1)} &   \alpha_{k+1}^{(1)} \\ 
\vdots & \vdots & \vdots  & \vdots & \vdots \\
\alpha  _{s_{1}}^{(1)} & \alpha  _{s_{1}+1}^{(1)}  & \ldots  &   \alpha_{s_{1}+k-2}^{(1)} &   \alpha_{s_{1}+k-1}^{(1)}  \\ 
\hline \\
\alpha  _{1}^{(2)} & \alpha  _{2}^{(2)}  & \ldots  &   \alpha_{k-1}^{(2)} &   \alpha_{k}^{(2)}  \\
\alpha  _{2}^{(2)} & \alpha  _{3}^{(2)}  & \ldots  &   \alpha_{k}^{(2)} &   \alpha_{k+1}^{(2)} \\ 
\vdots & \vdots & \vdots & \vdots & \vdots  \\
\alpha  _{s_{2}}^{(2)} & \alpha  _{s_{2}+1}^{(2)}  & \ldots  &   \alpha_{s_{2}+k-2}^{(2)} &   \alpha_{s_{2}+k-1}^{(2)}  \\ 
\hline\\
\alpha  _{1}^{(3)} & \alpha  _{2}^{(3)}  & \ldots  &   \alpha_{k-1}^{(3)}  &   \alpha_{k}^{(3)} \\
\alpha  _{2}^{(3)} & \alpha  _{3}^{(3)}  & \ldots  &   \alpha_{k}^{(3)}&   \alpha_{k+1}^{(3)}  \\ 
\vdots & \vdots & \vdots & \vdots & \vdots \\
\alpha  _{s_{3}}^{(3)} & \alpha  _{s_{3}+1}^{(3)}  & \ldots  &   \alpha_{s_{3}+k-2}^{(3)}&   \alpha_{s_{3}+k-1}^{(3)}   \\ 
\hline \\
\vdots & \vdots & \vdots  & \vdots & \vdots \\
\hline \\
\alpha  _{1}^{(m)} & \alpha  _{2}^{(m)}  & \ldots  &   \alpha_{k-1}^{(m)} &   \alpha_{k}^{(m)}  \\
\alpha  _{2}^{(m)} & \alpha  _{3}^{(m)}  & \ldots  &   \alpha_{k}^{(m)}&   \alpha_{k+1}^{(m)}  \\ 
\vdots & \vdots & \vdots  & \vdots & \vdots \\
\alpha  _{s_{m}}^{(m)} & \alpha  _{s_{m}+1}^{(m)} & \ldots   &   \alpha_{s_{m}+k-2}^{(m)} &   \alpha_{s_{m}+k-1}^{(m)}  \\ 
\end{array} \right )  
\end{equation}
$\boxed{  \textbf{The case $s_{1}=s_{2}=\ldots = s_{m}=1,\; m=\delta \leqslant k $}}  $  \\[0.05cm]
  
  Then $\delta = m$ and $ \Gamma_{\delta}^{\left[1\atop{\vdots \atop 1}\right]\times k} $denotes the number  of rank $ m$ $m\times k$  matrices over  $\mathbb{F}_{2} $  of the below form

\begin{equation}
 \label{eq 1.2} 
   \left (  \begin{array} {ccccc}
\alpha  _{1}^{(1)} & \alpha  _{2}^{(1)} & \ldots  &   \alpha_{k-1}^{(1)} &   \alpha_{k}^{(1)}  \\
\alpha  _{1}^{(2)} & \alpha  _{2}^{(2)}  & \ldots  &   \alpha_{k-1}^{(2)} &   \alpha_{k}^{(2)}  \\
\alpha  _{1}^{(3)} & \alpha  _{2}^{(3)}  & \ldots  &   \alpha_{k-1}^{(3)}  &   \alpha_{k}^{(3)} \\
\vdots & \vdots & \vdots  & \vdots & \vdots \\
\alpha  _{1}^{(m)} & \alpha  _{2}^{(m)}  & \ldots  &   \alpha_{k-1}^{(m)} &   \alpha_{k}^{(m)}  
\end{array} \right )  
\end{equation}
$\boxed{ \textbf{The case $m=1,\quad 2\leqslant s_{1}=\delta \leqslant k$}} $\\[0.05cm]
   Then $\delta = s_{1}$ and $ \Gamma_{\delta}^{[ s_{1}] \times k} $denotes the number  of persymmetric     $s_{1}\times k$  rank $ s_{1}$   matrices over  $\mathbb{F}_{2} $  of the below form
   
 \begin{equation}
 \label{eq 1.3}  
   \left (  \begin{array} {ccccc}
\alpha  _{1}^{(1)} & \alpha  _{2}^{(1)} & \ldots  &   \alpha_{k-1}^{(1)} &   \alpha_{k}^{(1)}  \\
\alpha  _{2}^{(1)} & \alpha  _{3}^{(1)}  & \ldots   &   \alpha_{k}^{(1)} &   \alpha_{k+1}^{(1)} \\ 
\vdots & \vdots & \vdots  & \vdots & \vdots \\
\alpha  _{s_{1}}^{(1)} & \alpha  _{s_{1}+1}^{(1)}  & \ldots  &   \alpha_{s_{1}+k-2}^{(1)} &   \alpha_{s_{1}+k-1}^{(1)}  \\ 
\end{array} \right )  
\end{equation}

$\boxed{   \textbf{The case $m=2,\quad 2\leqslant s_{1} \leqslant s_{2} ,\; s_{1}+s_{2}=\delta \leqslant k$}} $\\[0.05cm]
 
   Then $\delta = s_{1}+s_{2}$ and $ \Gamma_{\delta}^{\left[s_{1}\atop s_{2} \right]\times k} $denotes  the number  of double persymmetric \\

       $(s_{1}+s_{2}) \times k$  rank $ s_{1}+s_{2} $   matrices over  $\mathbb{F}_{2} $  of the below form

  \begin{equation}
 \label{eq 1.4}     
   \left (  \begin{array} {ccccc}
\alpha  _{1}^{(1)} & \alpha  _{2}^{(1)} & \ldots  &   \alpha_{k-1}^{(1)} &   \alpha_{k}^{(1)}  \\
\alpha  _{2}^{(1)} & \alpha  _{3}^{(1)}  & \ldots   &   \alpha_{k}^{(1)} &   \alpha_{k+1}^{(1)} \\ 
\vdots & \vdots & \vdots  & \vdots & \vdots \\
\alpha  _{s_{1}}^{(1)} & \alpha  _{s_{1}+1}^{(1)}  & \ldots  &   \alpha_{s_{1}+k-2}^{(1)} &   \alpha_{s_{1}+k-1}^{(1)}  \\ 
\hline \\
\alpha  _{1}^{(2)} & \alpha  _{2}^{(2)}  & \ldots  &   \alpha_{k-1}^{(2)} &   \alpha_{k}^{(2)}  \\
\alpha  _{2}^{(2)} & \alpha  _{3}^{(2)}  & \ldots  &   \alpha_{k}^{(2)} &   \alpha_{k+1}^{(2)} \\ 
\vdots & \vdots & \vdots & \vdots & \vdots  \\
\alpha  _{s_{2}}^{(2)} & \alpha  _{s_{2}+1}^{(2)}  & \ldots  &   \alpha_{s_{2}+k-2}^{(2)} &   \alpha_{s_{2}+k-1}^{(2)}  \\ 
\end{array} \right )  
\end{equation}

$\boxed{   \textbf{The case $m=3,\quad 2\leqslant s_{1} \leqslant s_{2} \leqslant s_{3} ,\; s_{1}+s_{2}+s_{3} = \delta \leqslant k$}} $\\[0.05cm]

   Then $\delta = s_{1}+s_{2}+s_{3} $ and $ \Gamma_{\delta}^{\left[ s_{1}\atop{s_{2} \atop s_{3}} \right]\times k} $denotes  the number  of triple  persymmetric \\
    $(s_{1}+s_{2}+s_{3}) \times k$  rank $ s_{1}+s_{2} +s_{3}$   matrices over  $\mathbb{F}_{2} $  of the below form
    
   \begin{equation}
 \label{eq 1.5}           
     \left (  \begin{array} {ccccc}
\alpha  _{1}^{(1)} & \alpha  _{2}^{(1)} & \ldots  &   \alpha_{k-1}^{(1)} &   \alpha_{k}^{(1)}  \\
\alpha  _{2}^{(1)} & \alpha  _{3}^{(1)}  & \ldots   &   \alpha_{k}^{(1)} &   \alpha_{k+1}^{(1)} \\ 
\vdots & \vdots & \vdots  & \vdots & \vdots \\
\alpha  _{s_{1}}^{(1)} & \alpha  _{s_{1}+1}^{(1)}  & \ldots  &   \alpha_{s_{1}+k-2}^{(1)} &   \alpha_{s_{1}+k-1}^{(1)}  \\ 
\hline \\
\alpha  _{1}^{(2)} & \alpha  _{2}^{(2)}  & \ldots  &   \alpha_{k-1}^{(2)} &   \alpha_{k}^{(2)}  \\
\alpha  _{2}^{(2)} & \alpha  _{3}^{(2)}  & \ldots  &   \alpha_{k}^{(2)} &   \alpha_{k+1}^{(2)} \\ 
\vdots & \vdots & \vdots & \vdots & \vdots  \\
\alpha  _{s_{2}}^{(2)} & \alpha  _{s_{2}+1}^{(2)}  & \ldots  &   \alpha_{s_{2}+k-2}^{(2)} &   \alpha_{s_{2}+k-1}^{(2)}  \\ 
\hline\\
\alpha  _{1}^{(3)} & \alpha  _{2}^{(3)}  & \ldots  &   \alpha_{k-1}^{(3)}  &   \alpha_{k}^{(3)} \\
\alpha  _{2}^{(3)} & \alpha  _{3}^{(3)}  & \ldots  &   \alpha_{k}^{(3)}&   \alpha_{k+1}^{(3)}  \\ 
\vdots & \vdots & \vdots & \vdots & \vdots \\
\alpha  _{s_{3}}^{(3)} & \alpha  _{s_{3}+1}^{(3)}  & \ldots  &   \alpha_{s_{3}+k-2}^{(3)}&   \alpha_{s_{3}+k-1}^{(3)}   \\ 
\end{array} \right )  
\end{equation}
   $ \boxed{ \textbf{The case $m\geqslant 4,\quad s_{1}=s_{2}=\ldots = s_{m-3}=1,\; 2 \leqslant s_{m-2} \leqslant s_{m-1} \leqslant s_{m},\; \sum_{j=1}^{m} s_{j}= \delta \leqslant k $}} $
   
      Then $\delta = m-3 + s_{m-2}+s_{m-1}+s_{m} $ and     $ \Gamma_{\delta}^{\left[1\atop {\vdots \atop{1\atop {s_{m-2}\atop{s_{m-1} \atop s_{m}}}}}\right]\times k} $  denotes  the number  of $ m-times $  persymmetric 
    $\delta \times k$  rank $ \delta$   matrices over  $\mathbb{F}_{2} $  of the below form

     \begin{equation}
 \label{eq 1.6}            
    \left (  \begin{array} {ccccc}
\alpha  _{1}^{(1)} & \alpha  _{2}^{(1)} & \ldots  &   \alpha_{k-1}^{(1)} &   \alpha_{k}^{(1)}  \\
\alpha  _{1}^{(2)} & \alpha  _{2}^{(2)}  & \ldots  &   \alpha_{k-1}^{(2)} &   \alpha_{k}^{(2)}  \\
\alpha  _{1}^{(3)} & \alpha  _{2}^{(3)}  & \ldots  &   \alpha_{k-1}^{(3)}  &   \alpha_{k}^{(3)} \\
\vdots & \vdots & \vdots  & \vdots & \vdots \\
\alpha  _{1}^{(m-3)} & \alpha  _{2}^{(m-3)}  & \ldots  &   \alpha_{k-1}^{(m-3)} &   \alpha_{k}^{(m-3)} \\
\hline \\    
  \alpha  _{1}^{(m-2)} & \alpha  _{2}^{(m-2)} & \ldots  &   \alpha_{k-1}^{(m-2)} &   \alpha_{k}^{(m-2)}  \\
\alpha  _{2}^{(m-2)} & \alpha  _{3}^{(m-2)}  & \ldots   &   \alpha_{k}^{(m-2)} &   \alpha_{k+1}^{(m-2)} \\ 
\vdots & \vdots & \vdots  & \vdots & \vdots \\
\alpha  _{s_{m-2}}^{(m-2)} & \alpha  _{s_{m-2}+1}^{(m-2)}  & \ldots  &   \alpha_{s_{m-2}+k-2}^{(m-2)} &   \alpha_{s_{m-2}+k-1}^{(m-2)}  \\ 
\hline \\
\alpha  _{1}^{(m-1)} & \alpha  _{2}^{(m-1)}  & \ldots  &   \alpha_{k-1}^{(m-1)} &   \alpha_{k}^{(m-1)}  \\
\alpha  _{2}^{(m-1)} & \alpha  _{3}^{(m-1)}  & \ldots  &   \alpha_{k}^{(m-1)} &   \alpha_{k+1}^{(m-1)} \\ 
\vdots & \vdots & \vdots & \vdots & \vdots  \\
\alpha  _{s_{m-1}}^{(m-1)} & \alpha  _{s_{m-1}+1}^{(m-1)}  & \ldots  &   \alpha_{s_{m-1}+k-2}^{(m-1)} &   \alpha_{s_{m-1}+k-1}^{(m-1)}  \\ 
\hline\\
\alpha  _{1}^{(m)} & \alpha  _{2}^{(m)}  & \ldots  &   \alpha_{k-1}^{(m)}  &   \alpha_{k}^{(m)} \\
\alpha  _{2}^{(m)} & \alpha  _{3}^{(m)}  & \ldots  &   \alpha_{k}^{(m)}&   \alpha_{k+1}^{(m)}  \\ 
\vdots & \vdots & \vdots & \vdots & \vdots \\
\alpha  _{s_{m}}^{(m)} & \alpha  _{s_{m}+1}^{(m)}  & \ldots  &   \alpha_{s_{m}+k-2}^{(m)}&   \alpha_{s_{m}+k-1}^{(m)}   \\ 
\end{array} \right )  
\end{equation}
\newpage
   \section{A conjecture concerning m-times persymmetric matrices}
  \label{sec 2}  
\textbf{Conjecture} : 
Let   $1\leqslant m\leqslant \sum_{j=1}^{m} s_{j}= \delta \leqslant k.$\\
We state that the number   $ \Gamma_{\delta}^{\left[s_{1}\atop{\vdots \atop s_{m}}\right]\times k} $ of rank $\delta $ m-times persymmetric matrices over the finite field   $ \mathbb{F}_{2}$ of the form \eqref{eq 1.1} is equal to \\
\begin{equation}
\label{eq 2.1}
 \displaystyle 2^{\delta -m}\prod_{j=1}^{m}(2^{k}-2^{\delta -j}) = 2^{(1+m)\delta-\frac{m^2}{2} -\frac{3m}{2}}\prod_{j=1}^{m}(2^{k-\delta+j} -1)
\end{equation}
In the case $\delta =k$ we deduce easily from \eqref{eq 2.1} that the fraction of invertible m-times persymmetric matrices is equal to  $\prod_{j=1}^{m}(1-2^{-j}) .$\\
To justify our assumption  \eqref{eq 2.1} we shall prove that our supposition is valid in the following  five cases:\\

\begin{itemize}
\item (1)  \textbf{The case $s_{1}=s_{2}=\ldots = s_{m}=1,\; m=\delta \leqslant k $} \\
\item  (2) \textbf{The case $m=1,\quad 2\leqslant s_{1}=\delta \leqslant k$}\\
\item (3)  \textbf{The case $m=2,\quad 2\leqslant s_{1} \leqslant s_{2} ,\; s_{1}+s_{2}=\delta \leqslant k$}\\
\item  (4)  \textbf{The case $m=3,\quad 2\leqslant s_{1} \leqslant s_{2} \leqslant s_{3} ,\; s_{1}+s_{2}+s_{3} = \delta \leqslant k$}\\
\item  (5)    \textbf{The case $m\geqslant 4,\quad s_{1}=s_{2}=\ldots = s_{m-3}=1,\; 2 \leqslant s_{m-2} \leqslant s_{m-1} \leqslant s_{m},\; \sum_{j=1}^{m} s_{j}= \delta \leqslant k $}
\end{itemize}

    \section{Justification of our  assumption  \eqref{eq 2.1} in the first case}
  \label{sec 3}  
 \begin{thm}
\label{thm 3.1}
The number of $m\times k$ matrices over $ \mathbb{F}_{2}$ of rank m  is equal to 
$\displaystyle \prod_{l=0}^{m-1}(2^{k}-2^{l}) = \prod_{j=1}^{m}(2^{k}-2^{m-j})$
\end{thm}
\begin{proof}
 See \eqref{eq 1.2} refer to  Section \ref{sec 1}.\\
We have,  $ \displaystyle  \Gamma_{m}^{\left[1\atop{\vdots \atop 1}\right]\times k} 
=  \prod_{l=0}^{m-1}(2^{k}-2^{l}) $\quad  (See George Landsberg [1] or S.D. Fisher and M.N Alexander [2] )
\end{proof}

  \section{Justification of our  assumption  \eqref{eq 2.1} in the second  case}
  \label{sec 4}  

 \begin{thm}
\label{thm 4.1}
The number of $\delta\times k$ persymmetric  matrices over $ \mathbb{F}_{2}$ of rank $\delta = s_{1}$  is equal to 
$\displaystyle 2^{k+\delta-1} -2^{2\delta-2} = 2^{\delta-1}\prod_{j=1}^{1}(2^{k}-2^{\delta-j})$
\end{thm}
\begin{proof}
 See \eqref{eq 1.3} refer to  Section \ref{sec 1}.\\
We have, \\
 $\displaystyle  \Gamma_{\delta}^{[ s_{1}] \times k}=  2^{k+\delta-1} -2^{2\delta-2}$\;(See David E. Daykin [3] or Cherly [4] )
 \end{proof}

  \section{Justification of our  assumption  \eqref{eq 2.1} in the third  case}
  \label{sec 5}  

 \begin{thm}
\label{thm 5.1}
The number of $\delta\times k$ double  persymmetric  matrices over $ \mathbb{F}_{2}$ of rank $\delta = s_{1}+s_{2}$  is equal to 
$\displaystyle   2^{\delta-2}\prod_{j=1}^{2}(2^{k}-2^{\delta-j})$
\end{thm}
\begin{proof}
 See \eqref{eq 1.4} refer to   Section \ref{sec 1}.\\
We have, \\
$\displaystyle  \Gamma_{\delta}^{\left[s_{1}\atop s_{2} \right]\times k} = 2^{2k+\delta-2}-3\cdot2^{k+2\delta-4}+2^{3\delta-5}=
2^{\delta-2}(2^{2k}-3\cdot2^{k+\delta-2}+2^{2\delta-3})=  2^{\delta-2}\prod_{j=1}^{2}(2^{k}-2^{\delta-j})$ (See Cherly [5] )
 \end{proof}

  \section{Justification of our  assumption  \eqref{eq 2.1} in the fourth case}
  \label{sec 6}  

 \begin{thm}
\label{thm 6.1}
The number of $\delta\times k$ triple  persymmetric  matrices over $ \mathbb{F}_{2}$ of rank $\delta = s_{1}+s_{2}+s_{3} $  is equal to 
$\displaystyle   2^{\delta-3}\prod_{j=1}^{3}(2^{k}-2^{\delta-j})$
\end{thm}
\begin{proof}
 See \eqref{eq 1.5} refer to  Section \ref{sec 1}.\\
We have, \\
$\displaystyle  \Gamma_{\delta}^{\left[ s_{1}\atop{s_{2} \atop s_{3}} \right]\times k} =  2^{3k+\delta-3}-7\cdot2^{2k+2\delta-6}+7\cdot2^{k+3\delta-8}+2^{4\delta-9}= 2^{\delta-3}(2^{3k}-7\cdot2^{2k+2\delta-3}+7\cdot2^{k+2\delta-5} -2^{3\delta-6})= 2^{\delta-3}\prod_{j=1}^{3}(2^{k}-2^{\delta-j})$ (See Cherly [6] )
 \end{proof}
  \section{Justification of our  assumption  \eqref{eq 2.1} in the fifth case}
  \label{sec 7}  

 \begin{thm}
\label{thm 7.1}
The number of $\delta\times k$ m-times  persymmetric  matrices over $ \mathbb{F}_{2}$ of rank $\delta = m-3+ s_{1}+s_{2}+s_{3} $  is equal to 
$\displaystyle   2^{\delta-m}\prod_{j=1}^{m}(2^{k}-2^{\delta-j})$
\end{thm}
\begin{proof}
 See \eqref{eq 1.6} refer to Section \ref{sec 1}.\\
We obtain from Lemma 10.1 in [4] with some obviously modifications that 
\begin{equation}
\label{eq 7.1}
   \Gamma_{\delta}^{\left[1\atop {\vdots \atop{1\atop {s_{m-2}\atop{s_{m-1} \atop s_{m}}}}}\right]\times k} =  \Gamma_{\delta-m+3}^{\left[ s_{m-2}\atop{s_{m-1} \atop s_{m}} \right]\times k}  \prod_{j=1}^{m-3}(2^{k}-2^{\delta-j})  
\end{equation}
We get from  Theorem \ref{thm 6.1} that  
\begin{equation}
\label{eq 7.2} 
 \Gamma_{\delta-m+3}^{\left[ s_{m-2}\atop{s_{m-1} \atop s_{m}} \right]\times k}= 
 2^{\delta-m+3-3} \prod_{j=1}^{3}(2^{k}-2^{\delta-m+3-j})  =  2^{\delta-m} \prod_{j=m-2}^{m}(2^{k}-2^{\delta-j}) 
\end{equation}
Combining \eqref{eq 7.1} and \eqref{eq 7.2} we obtain 

$\displaystyle    \Gamma_{\delta}^{\left[1\atop {\vdots \atop{1\atop {s_{m-2}\atop{s_{m-1} \atop s_{m}}}}}\right]\times k} = 2^{\delta-m}\prod_{j=1}^{m}(2^{k}-2^{\delta-j})$
 \end{proof}
 For other related articles concerning persymmetric matrices over the finite field with two elements see Cherly    [7],[8] and [9].

  \section{An example}
  \label{sec 8}   
 
 Let $ m,k,\delta $ be rational integers such that $1\leqslant m \leqslant \delta \leqslant k. $\\
 Let $ \delta \equiv i (mod\; m ) $ that is $  \delta = i +s\cdot m $ such that $ 0\leqslant i \leqslant m-1. $  \\

 We denote by  $\textit{M}_{m,\delta,k}$  the  $ k\times \delta $  matrix over  $\mathbb{F}_{2} $  of the below form:\\

\begin{equation}
\label{eq 8.1}
  \left ( \begin{array} {ccccc|cccc}
\alpha _{1} & \alpha _{2} & \alpha _{3} &  \ldots & \alpha _{m}  &  \alpha _{m+1} & \alpha _{m+2} & \ldots &  \alpha _{\delta}\\
\alpha _{m+1 } & \alpha _{m+2} & \alpha _{m+3}&  \ldots  &  \alpha _{2m} &  \alpha _{2m+1}  &  \alpha _{2m+2}  & \ldots &  \alpha _{\delta+m}\\
\alpha _{2m+1 } & \alpha _{2m+2} & \alpha _{2m+3}&  \ldots  &  \alpha _{3m} &  \alpha _{3m+1}  &  \alpha _{3m+2}  & \ldots &  \alpha _{\delta+2m}\\
\vdots & \vdots & \vdots   &  \vdots  & \vdots  & \vdots  &  \vdots  & \vdots  & \vdots \\
\vdots & \vdots & \vdots    &  \vdots & \vdots & \vdots &  \vdots  & \vdots  & \vdots  \\
\alpha _{(m-1)m+1 } & \alpha _{(m-1)m+2} & \alpha _{(m-1)m+3}&  \ldots  &  \alpha _{m^2} &  \alpha _{m^2+1}  &  \alpha _{m^2+2}  & \ldots &  \alpha _{\delta+(m-1)m}\\
\hline \\
\alpha _{m^2+1 } & \alpha _{m^2+2} & \alpha _{m^2+3}&  \ldots  &  \alpha _{m^2+m} &  \alpha _{m^2+m+1}  &  \alpha _{m^2+m+2}  & \ldots &  \alpha _{\delta+m^2}\\
\alpha _{m^2+m+1 } & \alpha _{m^2+m+2} & \alpha _{m^2+m+3}&  \ldots  &  \alpha _{m^2+2m} &  \alpha _{m^2+2m+1}  &  \alpha _{m^2+2m+2}  & \ldots &  \alpha _{\delta+m^2+m}\\
\vdots & \vdots & \vdots   &  \vdots  & \vdots  & \vdots  &  \vdots  & \vdots  & \vdots \\
\vdots & \vdots & \vdots    &  \vdots & \vdots & \vdots &  \vdots  & \vdots  & \vdots  \\
\alpha _{(k-1)m+1 } & \alpha _{(k-1)m+2} & \alpha _{(k-1)m+3}&  \ldots  &  \alpha _{km} &  \alpha _{km+1}  &  \alpha _{km+2}  & \ldots &  \alpha _{\delta+(k-1)m}\\
\end{array}  \right). 
\end{equation}

The transpose of  $\textit{M}_{m,\delta,k}$  becomes after a rearrangement of the rows a m-times  persymmetric  $\delta \times k $  matrix
with   $s_{1}=s_{2}=\ldots =s_{m-i} = s $ and $s_{m-i+1}=s_{m-i+2}=\ldots s_{m} = s+1,$  refer to Section \ref{sec 1}.\\
We then deduce from the conjecture  \eqref{eq 2.1} that the number of rank $\delta $ matrices of the form \eqref{eq 8.1} is equal to 
$ \displaystyle 2^{\delta -m}\prod_{j=1}^{m}(2^{k}-2^{\delta -j}) $

\begin{example}
Let m=3,\; $\delta = 8$ and k=10.\\
 $ 8 \equiv 2 (mod\; 3 ) $ that is $  8 = 2 +2\cdot 3 ,$ so i=2 and s=2.\\
  Then  $\textit{M}_{3,8,10}$  denotes  the  $ 10\times 8 $  matrix over  $\mathbb{F}_{2} $  of the below form:\\
  \begin{equation}
\label{eq 8.2}  
  \left ( \begin{array} {ccc|ccccc}
\alpha _{1} & \alpha _{2} & \alpha _{3} &  \alpha _{4} & \alpha _{5} &  \alpha _{6} & \alpha _{7} &  \alpha _{8}\\
\alpha _{4 } & \alpha _{5} & \alpha _{6}&  \alpha _{7} &  \alpha _{8}  &  \alpha _{9}  &  \alpha _{10} &  \alpha _{11}\\
\alpha _{7} & \alpha _{8} & \alpha _{9} &  \alpha _{10} & \alpha _{11} &  \alpha _{12} & \alpha _{13} &  \alpha _{14}\\
\hline \\
\alpha _{10 } & \alpha _{11} & \alpha _{12}&  \alpha _{13} &  \alpha _{14}  &  \alpha _{15}  &  \alpha _{16} &  \alpha _{17}\\
\alpha _{13} & \alpha _{14} & \alpha _{15} &  \alpha _{16} & \alpha _{17} &  \alpha _{18} & \alpha _{19} &  \alpha _{20}\\
\alpha _{16 } & \alpha _{17} & \alpha _{18}&  \alpha _{19} &  \alpha _{20}  &  \alpha _{21}  &  \alpha _{22} &  \alpha _{23}\\
\alpha _{19} & \alpha _{20} & \alpha _{21} &  \alpha _{22} & \alpha _{23} &  \alpha _{24} & \alpha _{25} &  \alpha _{26}\\
\alpha _{22 } & \alpha _{23} & \alpha _{24}&  \alpha _{25} &  \alpha _{26}  &  \alpha _{27}  &  \alpha _{28} &  \alpha _{29}\\
\alpha _{25} & \alpha _{26} & \alpha _{27} &  \alpha _{28} & \alpha _{29} &  \alpha _{30} & \alpha _{31} &  \alpha _{32}\\
\alpha _{28 } & \alpha _{29} & \alpha _{30}&  \alpha _{31} &  \alpha _{32}  &  \alpha _{33}  &  \alpha _{34} &  \alpha _{35}\\
\end{array}  \right). 
\end{equation}

We then obtain that the number  $ \Gamma_{8}^{\left[ 2 \atop{ 3 \atop 3}\right]\times 10}  $   of rank $\delta = 8$ matrices of the form \eqref{eq 8.2} is equal to $ \displaystyle 2^{5} \prod_{j=1}^{3}(2^{10}-2^{8 -j}) = 2^{23}\cdot7\cdot15\cdot31 = 3255 \cdot 2^{23}.$

\end{example}

\end{document}